\documentclass[letterpaper, 10 pt, conference]{ieeeconf}  % Comment this line out
\IEEEoverridecommandlockouts                              % This command is only
\usepackage{lipsum}
\usepackage{amsmath} % general nice math
\usepackage{bbm}
\usepackage{amssymb}
\usepackage{mathtools}
\usepackage{amsthm}

\usepackage{todonotes}

\theoremstyle{definition}
\newtheorem{definition}{Definition}
\newtheorem{theorem}{Theorem}
\newtheorem{lemma}{Lemma}
\newtheorem{assumption}{Assumption}
\newtheorem{corollary}{Corollary}
\newtheorem{problem}{Problem}

\usepackage{hyperref}
\hypersetup{
    colorlinks=true,
    linkcolor=black,
    filecolor=black,      
    urlcolor=blue,
    }
\usepackage{graphicx}
\graphicspath{ {resources/} }
\usepackage{booktabs}
\usepackage{float}
\usepackage{cite}

\let\labelindent\relax % fix enumitem package
\usepackage{enumitem}

\title{\LARGE \bf
Energy-Optimal Motion Planning for Agents: Barycentric Motion and Collision Avoidance Constraints
}

\author{Logan E. Beaver, \emph{Student Member, IEEE}, Michael Dorothy,\\Christopher Kroninger, Andreas A. Malikopoulos, \emph{Senior Member, IEEE}% <-this % stops a space
\thanks{This research was supported by Combat Capabilities Development Command, Army Research Laboratory, MD, USA.}
	\thanks{Logan and Andreas are with the Department of Mechanical Engineering, University of Delaware, Newark, DE, USA.
	Michael and Christopher are with Combat Capabilities Development Command, Army Research Laboratory, MD, USA. (emails: \tt\small{lebeaver@udel.edu};  \tt\small{michael.r.dorothy.civ@mail.mil};  \tt\small{christopher.m.kroninger.civ@mail.mil}; \tt\small{andreas@udel.edu}.)}%
}

\begin{document}

\maketitle
%\thispagestyle{empty}
%\pagestyle{empty}

%%%%%%%%%%%%%%%%%%%%%%%%%%%%%%%%%%%%%%%%%%%%%%%%%%%%%%%%%%%%%%%%%%%%%%%%%%%%%%%%
\begin{abstract}

As robotic swarm systems emerge, it is increasingly important to provide strong guarantees on energy consumption and safety to maximize system performance.
One approach to achieve these guarantees is through constraint-driven control, where agents seek to minimize energy consumption subject to a set of safety and task constraints.
In this paper, we provide a sufficient and necessary condition for an energy-minimizing agent with integrator dynamics to have a continuous control input at the transition between unconstrained and constrained trajectories.
In addition, we present and analyze barycentric motion and collision avoidance constraints to be used in constraint-driven control of swarms.
\end{abstract}

%%%   START OF PAPER
%%%%%%%%%%%%%%%%%%%%%%%%%%%%%%%%%%%%%%%%%%%%%%%%%%%%%%%%%%%%%%%%%%%%%%%%%%%%%%%%
\section{Introduction}

Control of swarms systems is an emerging topic in the fields of controls and robotics.
Due to their adaptability and flexibility \cite{Oh2017}, swarm systems have attracted considerable attention in transportation \cite{chalaki2020experimental}, construction \cite{Lindsey2012ConstructionTeams}, and surveillance\cite{Corts2009} applications.
As we advance to experimental swarm testbeds \cite{Pickem2017TheTestbed, Rubenstein2012} and outdoor experiments \cite{Vasarhelyi2018OptimizedEnvironments}, it is critical to minimize the cost per agent in the swarm by considering energy-minimizing algorithms with strong guarantees on safety and performance.

Safety and performance have recently been explored in the context of control barrier functions \cite{Notomista2019Constraint-DrivenSystems}, where agents react to the environment in real-time to satisfy task and safety constraints.
In this paper, we seek to advance the state of the art in real-time optimal control by providing: (1) a necessary and sufficient condition for continuity of the control input of an energy-minimizing agent with integrator dynamics, (2) a barycentric motion constraint that guarantees an agent's arrival to a desired set in finite time, and (3) a system of equations involving only the state and control variables that guarantees optimality for the barycentric motion and collision avoidance constraints.

Our first contribution has been explored on a case-by-case basis \cite{Zhang2019JointGuarantees,Malikopoulos2020,Beaver2020BeyondFlocking}; however, there has been no general continuity result reported in the literature.
Our main result is applicable as a coarse high-level plan for many systems, particularly with applications in connected and automated vehicles \cite{Mahbub2020decentralized} and drone swarming \cite{Beaver2020AnFlocking}.
Our barycentric motion constraint is inspired by the distributed formation control law presented in \cite{Fathian2019RobustVehicles}. 
In particular, it was shown that swarms of agents constrained to move toward a target point, e.g., a barycenter, are very robust to noise and disturbances. This is partially due to the large space of feasible control inputs.
Due to this property, we believe barycentric motion is a natural approach to drive agents into a desired set under the constraint-driven paradigm.
%Particularly, we believe our formulation will almost entirely eliminate the need for slack variables in constraint driven control of swarms \cite{Beaver2020BeyondFlocking}.

The remainder of the paper is organized as follows. In Section \ref{sec:theoretical}, we provide our main result that gives sufficient and necessary conditions for our proposed agent to have a continuous control input. In Section \ref{sec:barycentric}, we propose the barycentric motion constraint, and in Section \ref{sec:solution}, we derive the corresponding optimal motion primitive. In Section \ref{sec:jump}, we derive the remaining optimality conditions in terms of the state and control of the agent, and in Section \ref{sec:fixedDistance}, we present the optimality conditions for the agent to constrain itself to the surface of a closed disk, and equivalently, collision avoidance. Finally, we draw concluding remarks and future work in Section \ref{sec:conclusion}.

\section{Main Result} \label{sec:theoretical}
Consider a dynamical system $\mathcal{S}(t)$ with $k$ states in $\mathbb{R}^n$ at time $t\in\mathbb{R}$,
\begin{equation} \label{eq:states}
    \mathcal{S}(t) = \{\mathbf{x}_1(t), \mathbf{x}_2(t), \dots, \mathbf{x}_k(t)\},
\end{equation}
where $\mathbf{x}_i(t)\in\mathbb{R}^n, ~i=1,\ldots,k, ~k\in\mathbb{N}.$ Let the system obey integrator dynamics, i.e.,
\begin{equation} \label{eq:k-dynamics}
    \dot{\mathbf{x}}_p(t) = \begin{cases}
    \mathbf{x}_{p+1}(t), & \text{ if } p\in\{1, 2, \dots, k-1\}, \\
    \mathbf{u}(t), & \text{ if } p = k,
    \end{cases}
\end{equation}
where $\mathbf{u}(t)\in\mathbb{R}^n$ is the bounded control input, i.e., $||\mathbf{u}(\cdot)|| < \infty$.
Let $e(t)$ be the rate of energy consumption of the system given by 
\begin{equation} \label{eq:energy}
    e(t) = \frac{1}{2}||\mathbf{u}(t)||^2,
\end{equation}
i.e., rate of energy consumption is proportional to the $\mathcal{L}^2$ norm of the control input.
Finally, we impose the following constraint,
\begin{equation} \label{eq:constraint}
    g(\mathbf{x}(t),t) \leq 0,
\end{equation}
where \eqref{eq:constraint} is a class $C^{q-1}$ function, and $q\in\mathbb{N}$ is the minimum number of time derivatives of $g(\mathbf{x}(t),t)$ required for $\mathbf{u}(t)$ to appear in $g^{(q)}(\mathbf{x}(t), t)$.

\begin{lemma} \label{lma:uniqueness}
Given real vectors $\mathbf{a}, \mathbf{b} \in\mathbb{R}^m$, the unique real solution to the equation $||\mathbf{a}||^2 + ||\mathbf{b}||^2 = 2\,\mathbf{a}\cdot\mathbf{b}$ is $\mathbf{a} = \mathbf{b}$.
\end{lemma}

\begin{proof}
By the definition of the dot product, $||\mathbf{a}||^2 + ||\mathbf{b}||^2 = 2||\mathbf{a}||\,||\mathbf{b}||\cos{\theta_{ab}}$, where $\theta_{ab}$ is the angle between the vectors $\mathbf{a}$ and $\mathbf{b}$.
Rearranging and applying the quadratic formula for $||\mathbf{a}||$ yields
\begin{equation} \label{eq:quadratic}
    ||\mathbf{a}|| = \frac{1}{2}\Big(2||\mathbf{b}||\cos{\theta_{ab}} \pm \sqrt{4||\mathbf{b}||^2\cos^2{\theta_{ab}} - 4||\mathbf{b}||^2}\Big).
\end{equation}
Since $\mathbf{a}\in\mathbb{R}^m$, $||\mathbf{a}||\in\mathbb{R}$, hence
\begin{equation}
    4||\mathbf{b}||^2(\cos^2{\theta_{ab}} - 1) \geq 0,
\end{equation}
which implies $\cos{\theta_{ab}} = \pm 1$.

If $\cos{\theta_{ab}} = 1$, then $\theta_{ab} = 0$ and $||\mathbf{a}|| = ||\mathbf{b}||$ by \eqref{eq:quadratic}, which implies $\mathbf{a} = \mathbf{b}$.

If $\cos{\theta_{ab}} = -1$, then $||\mathbf{a}|| = - ||\mathbf{b}||$ by \eqref{eq:quadratic}, which implies $\mathbf{a} = \mathbf{b} =0$.

\end{proof}

To compute the energy-optimal control input for the system we follow the standard procedure for constrained optimal control \cite{Bryson1975AppliedControl}. First we take $q-1$ time derivatives of \eqref{eq:constraint} to construct a vector of tangency conditions,
\begin{equation} \label{eq:tangency}
    \mathbf{N}(\mathbf{x}(t), t) = 
    \begin{bmatrix}
    g(\mathbf{x}(t), t) \\
    g^{(1)}(\mathbf{x}(t), t) \\
    ... \\
    g^{(q-1)}(\mathbf{x}(t), t)
    \end{bmatrix},
\end{equation}
where $g(\mathbf{x}(t), t)$ is a class $C^{q-1}$ function.
Next, we seek the optimal control input $\mathbf{u}(t)$ that minimizes the Hamiltonian,
\begin{equation} \label{eq:constrainedHamiltonian}
    H = \frac{1}{2}||\mathbf{u}(t)||^2 + \boldsymbol{\lambda}(t)\cdot\mathbf{f}(\mathbf{x}(t), \mathbf{u}(t)) + \mu(t)\,g^{(q)}(\mathbf{x}(t), t),
\end{equation}
where $\boldsymbol{\lambda}(t) \in\mathbb{R}^{k\times n}$ are the influence functions, $\mathbf{f}(\mathbf{x}(t), \mathbf{u}(t))$ are the integrator dynamics defined by \eqref{eq:k-dynamics},  and $\mu(t)\in\mathbb{R}_{\geq0}$ is a Lagrange multiplier where
\begin{equation} \label{eq:mu}
    \mu(t) \begin{cases}
    = 0, &\text{ if } g(\mathbf{x}(t), t) < 0, \\
    \geq 0, & \text{ if } g(\mathbf{x}(t), t) = 0.
    \end{cases}
\end{equation}

The optimal control input must satisfy $\frac{\partial H}{\partial \mathbf{u}} = 0$, thus the optimal unconstrained input is
\begin{equation} \label{eq:uUnconstrained}
    \mathbf{u}(t) = -\boldsymbol{\lambda}_k(t),
\end{equation}
where $\boldsymbol{\lambda}_k(t)$ is the influence function corresponding to the state $\mathbf{x}_k(t)$ and $k = |\mathcal{S}(t)|$.

\begin{theorem} \label{thm:continuity}
    Consider the dynamical system $\mathcal{S}(t)$ in \eqref{eq:states} with an energy cost  \eqref{eq:energy}, and a scalar functional constraint on the state trajectory $g(\mathbf{x}(t), t)$.
    Suppose $\mathcal{S}(t)$ transitions between the constrained and unconstrained cases at time $t_1$, and let $\mathbf{N}(\mathbf{x}(t), t)$ denote the tangency conditions \eqref{eq:tangency}.
    If there exists $t_2 > t_1$, $t_1, t_2\in\mathbb{R},$ such that $g(\mathbf{x}(t), t) = 0$ for all $t\in[t_1, t_2]$, and $\mathbf{N}_t(\mathbf{x}(t_1), t_1) = \mathbf{0}$, %and $\mathbf{N}(\mathbf{x}(t_1), t_1)$ is right differentiable at $t_1$ %in $[t_1, t_2]$, 
    then the optimal control input $\mathbf{u}(t)$ is continuous at the junction  $t_1$.
\end{theorem}

\begin{proof}
The jump conditions of the influence functions and Hamiltonian at time $t$ are
\begin{align}
    \boldsymbol{\lambda}^T(t^+) = \boldsymbol{\lambda}^T(t^-) - \boldsymbol{\pi}^T \frac{\partial \mathbf{N}}{\partial \mathbf{x}}\Big|_{t}, \label{eq:covectorJump}\\
    H(t^+) - H(t^-) = \boldsymbol{\pi}^T \frac{\partial \mathbf{N}}{\partial t}\Big|_{t}, \label{eq:hamiltonianJump}
\end{align}
where $t^-$ and $t^+$ correspond to the left and right limits of $t$, respectively, and $\boldsymbol{\pi}$ is a $q\times1$ vector of constant Lagrange multipliers.
Substituting \eqref{eq:constrainedHamiltonian} into \eqref{eq:hamiltonianJump} yields
\begin{equation} \label{eq:hamiltonianJumpExpanded}
    \frac{1}{2}||\mathbf{u}^+||^2 + \boldsymbol{\lambda}^+\cdot\mathbf{f}^+ - \frac{1}{2}||\mathbf{u}^-||^2 - \boldsymbol{\lambda}^-\cdot\mathbf{f}^- = \boldsymbol{\pi}^T\mathbf{N}_t,
\end{equation}
where the superscripts $^-$ and $^+$ correspond to variables evaluated at $t^-$ and $t^+$, respectively. Note that $\mu(t^-) = 0$ and $g^{(q)}(\mathbf{x}(t^+), t^+) = 0$, thus those terms do not appear in \eqref{eq:hamiltonianJumpExpanded}. Substituting \eqref{eq:covectorJump} into \eqref{eq:hamiltonianJumpExpanded} yields
\begin{align} \label{eq:hamiltonianJumpExpandedCovector}
    \frac{1}{2}||\mathbf{u}^+||^2 &+ \Big(\boldsymbol{\lambda}^- - (\boldsymbol{\pi}^T \mathbf{N}_x)^T\Big)\cdot\mathbf{f}^+ - \frac{1}{2}||\mathbf{u}^-||^2 - \boldsymbol{\lambda}^-\cdot\mathbf{f}^- \notag\\ &= \boldsymbol{\pi}^T\mathbf{N}_t.
\end{align}

Next, we simplify the influence functions in \eqref{eq:hamiltonianJumpExpandedCovector} using continuity of the states and \eqref{eq:uUnconstrained},
\begin{align}
    \boldsymbol{\lambda}^-\cdot\mathbf{f}^+ - \boldsymbol{\lambda^-}\cdot\mathbf{f}^- \notag
    &=(\boldsymbol{\lambda}_1^-\cdot\mathbf{x}_2 + ... + \boldsymbol{\lambda}_k^-\cdot\mathbf{u}^+) \notag\\
    &\quad- (\boldsymbol{\lambda}_1^-\cdot\mathbf{x}_2 + ... + \boldsymbol{\lambda}_k^-\cdot\mathbf{u}^-) \notag\\
    &=\boldsymbol{\lambda}_k^-\cdot\mathbf{u}^+ - \boldsymbol{\lambda}_k^-\cdot\mathbf{u}^- \notag\\
    &= -\mathbf{u}^-\cdot\mathbf{u}^+ + ||\mathbf{u}^-||^2. \label{eq:covectorSimplified}
\end{align}
Substituting \eqref{eq:covectorSimplified} into \eqref{eq:hamiltonianJumpExpandedCovector} and rearranging yields
\begin{equation} \label{eq:hamiltonianZeroForm}
    \frac{1}{2}||\mathbf{u}^+||^2 +  \frac{1}{2}||\mathbf{u}^-||^2 - \mathbf{u}^+\cdot\mathbf{u}^-  = \boldsymbol{\pi}^T\mathbf{N}_t + (\boldsymbol{\pi}^T\mathbf{N}_x)\mathbf{f}^+.
\end{equation}
By Lemma \ref{lma:uniqueness}, the control input $\mathbf{u}(t)$ is continuous at the junction if and only if the right hand side of \eqref{eq:hamiltonianZeroForm} is zero, hence we may formulate the equivalent condition,
\begin{equation} \label{eq:RHS}
    \boldsymbol{\pi}^T\Big( \mathbf{N}_t + \mathbf{N}_x\mathbf{f}^+  \Big) = 0.
\end{equation}
Since $\boldsymbol{\pi}$ is generally nonzero, we seek to satisfy
\begin{equation} \label{eq:condition}
    \mathbf{N}_t + \mathbf{N}_x \mathbf{f}^+ = \mathbf{0}.
\end{equation}

Since $g(\mathbf{x}(t), t) = 0$ for all $t\in[t_1, t_2]$, its $q$ derivatives exist and are equal to zero for all $t\in[t_1^+, t_2^-]$ \cite{Bryson1963OptimalSolutions}.
Thus, $\mathbf{N}(\mathbf{x}(t), t) = 0$ for all $t\in[t_1^+, t_2^-]$.
Equating the time derivative of \eqref{eq:tangency} at $t_1^+$ to \eqref{eq:condition} yields the equivalent condition
\begin{equation} \label{eq:condition2}
     \frac{d}{dt}\mathbf{N}(\mathbf{x}(t_1^+), t_1^+) = \mathbf{N}_t + \mathbf{N}_x \mathbf{f}^+ = \mathbf{0}.
\end{equation}
Eq. \eqref{eq:condition2} can be expanded into a system of equations for the right-hand derivative of each row $r$, namely
\begin{equation} \label{eq:finalCondition}
    \frac{d}{dt} g^{(r)^+} = \frac{\partial}{\partial t}g^{(r)^+} + \frac{\partial g^{(r)^+}}{\partial \mathbf{x}} \frac{d \mathbf{x}^+}{dt},
\end{equation}
where $r = 0, 1, 2, ..., q-1$.
Condition \eqref{eq:finalCondition} is the definition of right $q$-derivatives of $g(\mathbf{x}, t)$, which exist by our premise.
Thus, \eqref{eq:condition} is always satisfied and the control input $\mathbf{u}(t)$ is continuous.

The analysis for the system when exiting the constrained arc to the unconstrained arc is the same, and thus due to space limitations, the proof is omitted.
\end{proof}

%Intuitively, Theorem \ref{thm:continuity} implies that for any well-defined optimal control problem, under the assumed dynamics and objective, the control input will be continuous at any junction between unconstrained and constrained arcs.

\begin{corollary} \label{cor:continuous}
Consider the dynamical system described in Theorem \ref{thm:continuity}.
If the system is traveling along the trajectory imposed by the constraint $g(\mathbf{x}(t), t) = 0$, and the tangency conditions, $\mathbf{N}(\mathbf{x}(t), t)$, are discontinuous at some time $t_1$, then the control input at $\mathbf{u}(t_1)$ is continuous if and only if a feasible $\mathbf{u}(t_1)$ exists.
\end{corollary}

\begin{proof}
Let $t_1$ be the time where any element of $\mathbf{N}(\mathbf{x}(t_1), t_1)$ is discontinuous, while $\mathbf{N}(\mathbf{x}(t_1^-), t_1^-) = \mathbf{0}$.
Continuity in the system state implies that at least one row of $\mathbf{N}(\mathbf{x}(t_i^+), t_1^+)$ must be nonzero.
To satisfy $g(\mathbf{x}(t), t) = 0$ for $t>t_1$ requires an infinite impulse control input at $t_1$ \cite{Bryson1975AppliedControl}, which contradicts the boundedness of $\mathbf{u}(t)$.
Thus, if a feasible control input exists, the system must transition to an unconstrained arc at $t_1$, hence $\mathbf{u}(t_1)$ is continuous by Theorem \ref{thm:continuity}.
\end{proof}

Next, we present a case study for Theorem \ref{thm:continuity} and Corollary \ref{cor:continuous} for a double-integrator system in $\mathbb{R}^2$ under a barycentric motion constraint. This is presented in the context of constraint-driven optimal control, where the agent seeks to minimize energy consumption under a path constraint.

\section{Problem Formulation} \label{sec:barycentric}

As a step toward real-time constraint-driven optimal control of swarm systems, consider a single agent in $\mathbb{R}^2$ with double integrator dynamics
\begin{align}
    \dot{\mathbf{p}}(t) &= \mathbf{v}(t), \label{eq:pDynamics} \\
    \dot{\mathbf{v}}(t) &= \mathbf{u}(t), \label{eq:vDynamics}
\end{align}
where $\mathbf{p}$ and $\mathbf{v}$ are the agent's position and velocity, respectively. The state of the agent is given by
\begin{equation}
    \mathbf{x}(t) = \begin{bmatrix}
    \mathbf{p}(t),\\
    \mathbf{v}(t)
    \end{bmatrix}.
\end{equation}

Let $\mathbf{p}_r(t)$ be a time-varying reference position which the agent seeks to reach in finite time. We denote the relative distance between the agent and reference state as
\begin{equation}
	\mathbf{r}(t) = \mathbf{p}(t) - \mathbf{p}_r(t).
\end{equation}

\begin{definition} \label{def:barycentric}
    For a desired aggregation distance $D\in\mathbb{R}_{>0}$, we define the barycentric motion constraint
    \begin{equation} \label{eq:barycentricConstraint}
        g(\mathbf{x}(t), t) = \begin{cases}
        \beta(\mathbf{x}(t)) + \mathbf{r}(t)\cdot\dot{\mathbf{r}}(t), & ||\mathbf{r}(t)|| > D \\
        \mathbf{r}(t)\cdot\mathbf{r}(t) - D^2, & ||\mathbf{r}(t)|| \leq D
        \end{cases}
        \quad \leq 0,
    \end{equation}
    where $\beta(\cdot) > 0$ imposes decreasing barycentric motion toward the closed disk.% as shown in Fig. \ref{fig:barycentric}.
    We refer to $||\mathbf{r}(t)|| > D$ and $||\mathbf{r}(t)|| \leq D$ as Case I and Case II of $g(\mathbf{x}(t), t)$, respectively.
\end{definition}

%\begin{figure}[ht]
%    \centering
%    \includegraphics[width=0.4\linewidth]{Barycentric2.PNG}
%    \caption{A schematic showing the vectors used in Definition \ref{def:barycentric} between the agent (circle) and reference (rectangle). Imposing $\beta(\cdot) > 0$ implies $\theta > 0$ and imposes barycentric motion on the agent.}
%    \label{fig:barycentric}
%\end{figure}

Finally, we present a constraint-driven optimal control problem to determine the energy-optimal trajectory of the agent under the barycentric motion constraint.

\begin{problem} \label{prb:optimalControl}
The problem is formulated as follows:
\begin{align}
    \min_{\mathbf{u}(t)}~ & \frac{1}{2} \int_{t^0}^{t^f} ||\mathbf{u}(t)||^2 ~ dt\\
    \text{subject to:}~ & \eqref{eq:pDynamics}, \eqref{eq:vDynamics}, \eqref{eq:barycentricConstraint}, \\
    \text{given:} ~ &\mathbf{p}(t^0), \mathbf{v}(t^0).
\end{align}
where $[t^0, t^f]\subset\mathbb{R}$ is the planning horizon for the agent.
\end{problem}

To solve Problem \ref{prb:optimalControl} we impose the following assumptions.
\begin{assumption} \label{smp:feasible}
The initial state of the agent, $\mathbf{x}(t^0)$, satisfies \eqref{eq:barycentricConstraint}.
\end{assumption}

We impose Assumption \ref{smp:feasible} to ensure that the agent can generate a feasible trajectory at time $t^0$. This assumption may be relaxed by extending \eqref{eq:barycentricConstraint} to include an additional case.
However, this would add complexity to the problem without fundamentally changing our analysis.

\begin{assumption} \label{smp:perfect}
There are no disturbances or noise, and the agent is able to track the trajectory generated by Problem \ref{prb:optimalControl}.
\end{assumption}

We impose Assumption \ref{smp:perfect} to analyze the agent's behavior in a deterministic setting.
This assumption may be relaxed by imposing some notion of robustness to Problem \ref{prb:optimalControl} or reformulating it as a stochastic optimal control problem.

\begin{assumption}\label{smp:uddot}
The reference trajectory satisfies $\frac{d^4\mathbf{p}_r}{dt^4} = 0$.
\end{assumption}

Assumption \ref{smp:uddot} simplifies the analysis in Section \ref{sec:solution}. This assumption may be relaxed by carrying the term $\frac{d^4\mathbf{p}_r}{dt^4}$ through the final steps, which adds complexity without a significant impact on our results.

\section{Solution Approach} \label{sec:solution}

As a first step, we prove the continuity of control when transitioning between the cases of the barycentric motion constraint (Definition \ref{def:barycentric}). Then, we discuss the constrained motion primitive for Problem \ref{prb:optimalControl} and show how to optimally piece together the unconstrained and constrained arcs using only state and control conditions.

\subsection{Properties of the Barycentric Motion Constraint}
The constraint $g(\mathbf{x}(t),t)$ has two cases  (see Definition \ref{def:barycentric}), Case I corresponds to a \emph{barycentric spiral}, and Case II corresponds to a closed disk centered on the reference state.
When an agent transitions between an unconstrained arc and the arc defined by $g(\mathbf{x}, t) = 0$, the control input $\mathbf{u}(t)$ is continuous by Theorem \ref{thm:continuity}.
Next, we present three lemmas that describe the behavior of the agent while traveling along $g(\mathbf{x}(t), t) = 0$.

\begin{lemma} \label{lma:finiteTime}
    If there exist $\gamma\in\mathbb{R}_{>0}$ that lower bounds $\beta(\mathbf{x}(t))$ for all $t\in\mathbb{R}$, then the agent will satisfy $||\mathbf{r}(t)|| \leq D$ in finite time.
\end{lemma}

\begin{proof}
    Let the agent satisfy $||\mathbf{r}(t)|| > D$.
    By Definition \ref{def:barycentric}, $\beta\big(\mathbf{x}(t)\big) + \mathbf{r}(t)\cdot\dot{\mathbf{r}}(t) \leq 0$.
    This implies that $\beta\big(\mathbf{x}(t)\big) + ||\mathbf{r}(t)||\, ||\dot{\mathbf{r}}(t)||\cos{\theta_r(t)} \leq 0$ by the definition of the dot product, where $\theta_r(t)$ is the angle between $\mathbf{r}(t)$ and $\dot{\mathbf{r}}(t)$.
    Substituting the lower bounds for $\beta\big(\mathbf{x}(t)\big)$ and $\mathbf{r}(t)$ and rearranging yields $||\dot{\mathbf{r}}(t)|| \cos{\theta_r(t)} < - \frac{\gamma}{D}$.
    Thus, the component of $\dot{\mathbf{r}}(t)$ in the direction of $\mathbf{r}(t)$ has a negative sign and is upper bounded by a negative constant.
    This implies that $||\mathbf{r}(t)||$ will decrease to a distance $D$ in finite time.
\end{proof}

\begin{lemma} \label{lma:disk}
If the agent enters the closed disk of diameter $D$, as described by Case II of Definition \ref{def:barycentric}, then the agent will remain within the disk for all time.
\end{lemma}

\begin{proof}
Consider the case that $||\mathbf{r}(t)|| = D$ in \eqref{eq:barycentricConstraint}.
To exit the disk at a time $t_1$, the agent must satisfy $\mathbf{r}(t_1)\cdot\dot{\mathbf{r}}(t_1) > 0$.
However, by continuity of $\mathbf{x}(t)$, there exists some $\epsilon > 0$ such that $||\mathbf{r}(t_1+\epsilon)|| > D$ and $\mathbf{r}(t_1+\epsilon)\cdot\dot{\mathbf{r}}(t_1+\epsilon) > 0$.
This is infeasible by Definition \ref{def:barycentric}, thus the agent will remain within the closed disk for all time.
\end{proof}

\begin{lemma} \label{lma:transition}
If the the agent is travelling along the constrained arc described by Definition \ref{def:barycentric}, and transitions from Case I to Case II at a time $t_1$ and distance $||\mathbf{r}(t_1)|| = D$, then the control input is continuous at the transition.
\end{lemma}

\begin{proof}
When the agent transitions from Case I to Case II at $t_1$, we have $\dot{\mathbf{r}}_i(t_1^-)\cdot\mathbf{r}_i(t_1^-) = -\beta\big(\mathbf{x}(t_1^-)\big) < 0$.
Continuity of $\mathbf{x}_i(t)$ and $\mathbf{x}_r(t)$ implies that $\dot{\mathbf{r}}_i(t_1^+)\cdot\mathbf{r}_i(t_1^+) = -\beta\big(\mathbf{x}(t^+)\big) < 0$.
To stay on the constrained arc requires that $\dot{\mathbf{r}}_i(t_1^+)\cdot\mathbf{r}_i(t_1^+) = 0$. Thus, agent $i$ must exit the constrained arc at $t_1$, and $\mathbf{u}_i(t_1)$ is continuous by Corrolary \ref{cor:continuous}.
\end{proof}

By Lemmas \ref{lma:finiteTime}--\ref{lma:transition} we have proven that our proposed barycentric motion constraint 1) will drive the agent within a disk of diameter $D$ in finite time, 2) the agent will remain within the disk for all future time, and 3) the discontinuity in $g\big(\mathbf{x}(t), t\big)$ does not introduce a discontinuity into the optimal control input.
Next, we describe the constrained motion primitive for Case I of the barycentric motion constraint.

\subsection{Constrained Motion Primitive}

To solve for the constrained motion of the agent when $||\mathbf{r}(t)|| > D$, we use Hamiltonian analysis  \cite{Bryson1975AppliedControl}. First, we construct the vector of tangency conditions,
\begin{equation} \label{eq:tangencyBarycentric}
  \mathbf{N}(\mathbf{x}(t), t) = \Big[ \beta(\mathbf{x}_i) + \mathbf{r}_i(t)\cdot\dot{\mathbf{r}}_i(t) \Big],
\end{equation}
and we append the derivative of \eqref{eq:tangencyBarycentric} to the Hamiltonian,
\begin{equation} \label{eq:constrainedHamiltonianBary}
H = \frac{1}{2}||\mathbf{u}||^2 + \boldsymbol{\lambda^p}\cdot\mathbf{v} + \boldsymbol{\lambda^v}\cdot\mathbf{u} + \mu \Big(
    \dot{\beta} + \mathbf{r}\cdot\ddot{\mathbf{r}} + \dot{\mathbf{r}}\cdot\dot{\mathbf{r}}
    \Big).
\end{equation}
The Euler-Lagrange equations are
\begin{align}
    \mathbf{u}(t) &= -\boldsymbol{\lambda}^v(t) - \mu(t)\Big( \dot{\beta}_u\big(\mathbf{x}(t)\big) + \mathbf{r}(t) \Big), \label{eq:el-u}\\
    -\dot{\boldsymbol{\lambda}}^v(t) &= \boldsymbol{\lambda}^p(t) + \mu(t)\Big( \dot{\beta}_v\big(\mathbf{x}(t)\big) + \dot{\mathbf{r}}(t) \Big), \label{eq:el-lv}\\
    -\dot{\boldsymbol{\lambda}}^p(t) &= \mu(t)\Big( \dot{\beta}_p\big(\mathbf{x}(t)\big) + \ddot{\mathbf{r}}(t) \Big).\label{eq:el-lp}
\end{align}

To solve \eqref{eq:el-u} - \eqref{eq:el-lp} we follow the method outlined in \cite{Beaver2020AnFlocking}. Since Problem \ref{prb:optimalControl} is a generalization of the problem reported in \cite{Beaver2020AnFlocking}, we impose that $||\dot{\mathbf{r}}(t)||$ is a constant. This is the \emph{reigning optimal} solution \cite{Ross2015} for this constrained motion primitive. 

\begin{definition} \label{def:systemBasis}
Consider the basis of $\mathbb{R}^2$ defined by the vectors
\begin{align}
    \hat{p}\big(\mathbf{x}(t)\big) &= \frac{\mathbf{r}(t)}{||\mathbf{r}(t)||} = \frac{\mathbf{r}(t)}{b(t)}, \label{eq:phat} \\
    \hat{q}\big(\mathbf{x}(t)\big) &= \frac{\dot{\mathbf{r}}(t)}{||\dot{\mathbf{r}}(t)||} = \frac{\dot{\mathbf{r}}(t)}{a}, \label{eq:qhat}
\end{align}
where $a = ||\dot{\mathbf{r}}(t)||$ and $b(t) = ||\mathbf{r}(t)||$. This is a well defined basis for $\mathbb{R}^2$ as long as $a \neq 0$, $b(t) \neq 0$, and $\beta \neq ||\mathbf{r}||\,||\dot{\mathbf{r}}||.$
\end{definition}

For simplicity we will omit the dependence of the unit vectors $\hat{p}\big(\mathbf{x}(t)\big)$ and $\hat{q}\big(\mathbf{x}(t)\big)$ on $\mathbf{x}(t)$ when no ambiguity arises.
To guarantee that the basis in Definition \ref{def:systemBasis} is always well defined we select the following functional form for $\beta$,
\begin{equation} \label{eq:beta}
    \beta\big(\mathbf{x}(t)\big) = a\,b(t)\,\kappa,
\end{equation}
where $\kappa\in(0, 1)$ is the cosine of the angle between $\mathbf{r}$ and $\dot{\mathbf{r}}$ by the definition of the dot product.
As $b(t) > D$, we impose $a > 0$ when traveling along the barycentric spiral.

Following the procedure of \cite{Beaver2020AnFlocking}, we may project $\ddot{\mathbf{r}}(t)$ onto the unit vectors $\hat{p}$ and $\hat{q}$, which yields
\begin{equation} \label{eq:rddot}
    \ddot{r}(t) \cdot \begin{bmatrix}
    \hat{p} \\ \hat{q}
    \end{bmatrix}
    = \begin{bmatrix}
    - \frac{a^2 + \dot{\beta}\big(\mathbf{x}(t)\big)}{b(t)} \\
    \dot{a}
    \end{bmatrix},
\end{equation}

We then seek the time derivatives of \eqref{eq:phat} and \eqref{eq:qhat}. First,
\begin{align}
    \dot{\hat{p}} &= \frac{\dot{\mathbf{r}}(t)}{b(t)} - \frac{\dot{b}(t)}{b(t)^2}\mathbf{r}(t) \notag\\&= \frac{a}{b(t)}\hat{q} - \frac{\dot{b}(t)}{b(t)}\hat{p},
\end{align}
where
\begin{align} \label{eq:bdot}
    \dot{b}(t) &= \frac{d}{dt} ||\mathbf{p}_i(t) - \mathbf{p}_j(t)|| \notag\\
    &= \frac{\mathbf{p}_i(t) - \mathbf{p}_j(t)}{||\mathbf{p}_i(t) - \mathbf{p}_j(t)||}\cdot\Big(\mathbf{v}_i(t) - \mathbf{v}_j(t)\Big) \notag\\
    &= \frac{\mathbf{r}(t)\cdot\dot{\mathbf{r}}(t)}{b(t)} = \frac{-\beta\big(\mathbf{x}(t), t\big)}{b(t)} = -a \kappa,
\end{align}
thus,
\begin{equation} \label{eq:phatdot}
    \dot{\hat{p}} = \frac{a}{b(t)}\hat{q} + \frac{a}{b(t)}\kappa\,\hat{p}.
\end{equation}
It follows that
\begin{align}
    \dot{\hat{q}} = \frac{\ddot{\mathbf{r}}(t)}{a},
\end{align}
and substituting \eqref{eq:rddot} yields
\begin{align}
    \dot{\hat{q}} = -\frac{1}{ab(t)}\Big(a^2 + \dot{\beta}\big(\mathbf{x}(t)\big)\Big)\hat{p},
\end{align}
where
\begin{equation} \label{eq:betadot}
    \dot{\beta}\big(\mathbf{x}(t)\big) = a\dot{b}(t)\kappa = - (a\kappa)^2 = \dot{\beta}, 
\end{equation}
thus,
\begin{equation} \label{qhatdot}
    \dot{\hat{q}} = -\frac{a}{b(t)} \big(1-\kappa^2\big) \hat{p}.
\end{equation}
Substituting \eqref{eq:beta} into \eqref{eq:rddot} and by the definition of $\ddot{\mathbf{r}}$, we have
\begin{equation} \label{eq:initialRows}
    \big(\mathbf{u}_i(t) - \mathbf{u}_r(t)\big) \cdot \begin{bmatrix}
    \hat{p} \\ \hat{q}
    \end{bmatrix}
    = \begin{bmatrix}
    -\frac{a^2(1-\kappa^2)}{b(t)} \\ 0
    \end{bmatrix}.
\end{equation}
Next, we solve each row of \eqref{eq:initialRows} by substituting in the Euler-Lagrange equations and taking time derivatives until we have a system of ordinary differential equations that are only a function of $a$, $b(t)$, $\mu(t)$, and their derivatives.
We start by decomposing \eqref{eq:initialRows} into a system of two equations,
\begin{align}
    \big(\mathbf{u}_i(t) - \mathbf{u}_r(t)\big) \cdot\hat{p} &= - \frac{a^2}{b(t)}(1-\kappa^2),\\
    \big(\mathbf{u}_i(t) - \mathbf{u}_r(t)\big) \cdot\hat{q} &= 0,
\end{align}
where substituting \eqref{eq:el-u} and by rearranging we have
\begin{align}
    \Big(\boldsymbol{\lambda}^v(t) + \mathbf{u}_r(t) + \mu(t)\mathbf{r}(t)\Big)\cdot\hat{p} &= \frac{a^2}{b(t)}(1-\kappa^2), \label{eq:lambdava}\\
    \Big(\boldsymbol{\lambda}^v(t) + \mathbf{u}_r(t) + \mu(t)\mathbf{r}(t)\Big)\cdot\hat{q} &= 0, \label{eq:lambdavb}
\end{align}
which, by \eqref{eq:tangencyBarycentric}, simplifies to
\begin{align}
    \Big(\boldsymbol{\lambda}^v(t) + \mathbf{u}_r(t) \Big)\cdot\hat{p} &= \frac{a^2}{b(t)}(1-\kappa^2) - \mu(t)\,b(t), \label{eq:rows0a}\\
    \Big(\boldsymbol{\lambda}^v(t) + \mathbf{u}_r(t) \Big)\cdot\hat{q} &= \mu(t)\,b(t)\,\kappa. \label{eq:rows0b}
\end{align}
The next step is to take a time derivative of \eqref{eq:rows0a} and \eqref{eq:rows0b}, then substitute \eqref{eq:rows0a} and \eqref{eq:rows0b} in for the $\hat{p}$ and $\hat{q}$ terms that appear. Then we substitute \eqref{eq:el-lv} into the resulting equations and simplify, which yields 
\begin{align}
    \Big( \dot{\mathbf{u}}_r(t) - \boldsymbol{\lambda}^p(t)  \Big)\cdot\hat{p} &= 2a\mu(t)\kappa^2 - \dot{\mu}(t)b(t), \label{eq:lpp}\\
    \Big( \dot{\mathbf{u}}_r(t) - \boldsymbol{\lambda}^p(t)\Big)\cdot\hat{q} &= 
    \frac{a^3}{b(t)^2}(1-\kappa^2)^2 + \dot{\mu}(t)b(t)\kappa. \label{eq:lpq}
\end{align}
Finally, we take a time derivative of \eqref{eq:lpp} and \eqref{eq:lpq} and substitute \eqref{eq:el-lp}. Applying Assumption \ref{smp:uddot} and simplifying yields
\begin{align}
    \frac{a^4}{b(t)^3}\Big(1-\kappa^2\Big)^2 + \ddot{\mu}(t)b(t) = \frac{a^2}{b(t)}\mu(t)(1-\kappa^2) + \mu(t) a \kappa, \label{eq:ode1} \\
    a\dot{\mu}(t)(1-\kappa^2) + \dot{\mu}(t)a\kappa = \ddot{\mu}(t)b(t)\kappa + 2\frac{a^4}{b(t)^3}(1-\kappa^2)^2. \label{eq:ode2}
\end{align}
Equations \eqref{eq:ode1} and \eqref{eq:ode2} describe the evolution of $\dot{\mu}(t)$ and $b(t)$ for a given constant speed $a$ and barycentric parameter $\kappa$.

To find the optimal control input to the agent we may integrate \eqref{eq:bdot},
\begin{equation} \label{eq:b}
    b(t) = b^0 - a\kappa(t-t^0),
\end{equation}
where $b(t_1) = b^0$.
Finally, substituting \eqref{eq:betadot} and \eqref{eq:b} into \eqref{eq:rddot} yields
\begin{align}
    \ddot{\mathbf{r}}_i(t) = -\frac{a^2}{b_i^0-a\kappa(t-t^0)}(1 - \kappa^2) \hat{p} + 0 \hat{q},
\end{align}
which, by definition of the dot product, is
\begin{align}
    \ddot{\mathbf{r}}(t)\cdot\hat{p} = ||\ddot{\mathbf{r}}(t)||\,||\hat{p}||\cos(\frac{\pi}{2} - \theta_{pq}) = ||\ddot{\mathbf{r}}(t)|| \sin(\theta_{pq}) \notag\\ = ||\ddot{\mathbf{r}}(t)||\sin(\arccos(\kappa)) = ||\ddot{\mathbf{r}}(t)||\sqrt{1-\kappa^2}.
\end{align}
Thus,
\begin{equation}
    ||\ddot{\mathbf{r}}(t)|| = \frac{a^2}{b_i^0-a\kappa(t-t^0)} \sqrt{1-\kappa^2},
\end{equation}
and the orientation of $\ddot{\mathbf{r}}(t)$ is perpendicular to $\hat{q}$.
In the next subsection we use \eqref{eq:ode1} and \eqref{eq:ode2} to determine how the agent will optimally transition from the unconstrained to the constrained arc.

\subsection{Barycentric Jump Conditions} \label{sec:jump}

Let the agent transition from an unconstrained to barycentric-constrained arc at some time $t$. The jump conditions for the influence functions are \cite{Bryson1975AppliedControl},
\begin{align}
    \boldsymbol{\lambda}_p(t^-) &= \boldsymbol{\lambda}_p(t^+) + a\pi [\kappa\hat{p} + \hat{q}], \label{eq:lpJump}\\
    \boldsymbol{\lambda}_v(t^-) &= \boldsymbol{\lambda}_v(t^+) + b(t)\pi [\hat{p} + \kappa\hat{q}]. \label{eq:lvJump}
\end{align}
Substituting \eqref{eq:el-u} into \eqref{eq:lvJump} and applying continuity of $\mathbf{u}_i(t)$ implies
\begin{align}
    \mu(t^+)\hat{p} = \pi \big[ \hat{p} + \hat{q}\kappa  \big]. \label{eq:muPhat}
\end{align}
Next, we project \eqref{eq:muPhat} onto the unit vectors $\hat{p}$ and $\hat{q}$ which yields two scalar equations,
\begin{align}
    \mu(t^+) &= \pi \big [1 - \kappa^2], \\
    -\kappa\,\mu(t^+) &= \pi \big[-\kappa + \kappa \big].
\end{align}
By definition $\kappa\in(0, 1)$, which implies that $\mu(t^+) = 0$ and $\pi = 0$.
Thus we may simplify \eqref{eq:lpJump} to
\begin{equation}
    \boldsymbol{\lambda}_p(t^-) = \boldsymbol{\lambda}_p(t^+).
\end{equation}
Finally, the time derivatives of $\mathbf{u}_i(t)$ are
\begin{align}
    \dot{\mathbf{u}}_i(t^-) &= \boldsymbol{\lambda}_p(t^-), \\ 
    \dot{\mathbf{u}}_i(t^+) &= \boldsymbol{\lambda}_p(t^+) + \dot{\mu}(t^+)\,\mathbf{r}(t),
\end{align}
thus
\begin{equation}
    \dot{\mathbf{u}}(t^+) - \dot{\mathbf{u}}(t^-) = \dot{\mu}(t^+)\,\mathbf{r}(t).
\end{equation}
We may substitute $\mu(t^+) = 0$ into \eqref{eq:ode1} and \eqref{eq:ode2}, which yields
\begin{equation} \label{eq:mutPlus}
    \dot{\mu}(t^+) = \frac{a^3}{b(t)^3}\kappa\frac{\Big(1-\kappa^2\Big)^2}{1-\kappa^2 +\kappa}.
\end{equation}

Thus,
\begin{align} \label{eq:optimality2}
    \dot{\mathbf{u}}(t^+) - \dot{\mathbf{u}}(t^-) &= \mathbf{r}(t)\Bigg(\frac{a^3}{b(t)^3}\kappa\frac{\Big(1-\kappa^2\Big)^2}{1-\kappa^2 +\kappa}\Bigg).
\end{align}
At the junction between the unconstrained and constrained cases, we have four unknowns, $\mathbf{p}(t), ||\mathbf{v}(t)||$, and the optimal transition time $t$. The corresponding four equations are continuity in $\mathbf{u}(t)$, by Theorem \ref{thm:continuity}, and \eqref{eq:optimality2}.

We may then apply \eqref{eq:tangencyBarycentric} to find $\mathbf{v}(t)$ at the transition, which gives us sufficient boundary conditions to solve for the unconstrained and constrained trajectories.
%This enables the agent to generate an energy-optimal trajectory using only conditions on the state and control at junctions between motion primitives.
In the next section we provide the optimality conditions when activating the disk constraint given in Case II of Definition \ref{def:barycentric}, which also applies to collision avoidance.%. We also extend the analysis to include collision avoidance constraints with a fixed safe distance.

\subsection{Fixed Distance Constraint} \label{sec:fixedDistance}

Next we consider when $||\mathbf{r}(t)|| \leq D$.
The corresponding tangency conditions are
\begin{align} \label{eq:tangencyFixed}
    \mathbf{N}(\mathbf{x}, t) = \begin{bmatrix}
    \mathbf{r}(t)\cdot\mathbf{r}(t) - D^2\\
    2\mathbf{r}(t)\cdot\dot{\mathbf{r}}(t)
    \end{bmatrix}, \\
    g^{(2)}\big(\mathbf{x}(t), t\big) = 2\mathbf{r}(t)\cdot\ddot{\mathbf{r}}(t) + 2\dot{\mathbf{r}}(t)\cdot\dot{\mathbf{r}}(t),
\end{align}
which leads to the same analysis as Section \ref{sec:barycentric} when $\mathbf{N}\big(\mathbf{x}(t), t) = \mathbf{0}$ and $g^{(2)}\big(\mathbf{x}(t), t\big) = 0$ if we impose $b(t)=D$ and $\kappa = 0$. In this case \eqref{eq:ode2} implies that $\mu(t)$ is a constant, and \eqref{eq:ode1} implies that
\begin{equation}
    \mu = \Big(\frac{a}{D}\Big)^2.
\end{equation}

The derivative of $\mathbf{N}$ with respect to the state is
\begin{equation}
    \frac{\partial\mathbf{N}(\mathbf{x}(t), t)}{\partial{\mathbf{x}}(t)} = \begin{bmatrix}
    2\,\mathbf{r}(t)^T & 0 \\
    2\,\dot{\mathbf{r}}(t)^T & 2\,\mathbf{r}(t)^T
    \end{bmatrix},
\end{equation}
and the jump conditions at time $t$ are \cite{Bryson1975AppliedControl},
\begin{align}
    \boldsymbol{\lambda}_p(t^-) &= \boldsymbol{\lambda}_p(t^+) + 2\big[ \pi_1\mathbf{r}(t) + \pi_2\dot{\mathbf{r}}(t) \big], \label{eq:lpJumpS}\\
    \boldsymbol{\lambda}_v(t^-) &= \boldsymbol{\lambda}_v(t^+) + 2\big[ \pi_2\mathbf{r}(t),  \big] \label{eq:lvJumpS}
\end{align}
The condition $\frac{\partial H}{\partial u}|_{t^-} = \frac{\partial H}{\partial u}|_{t^+}$ implies
\begin{equation}
    \mathbf{u}(t^-) + \boldsymbol{\lambda}_v(t^-) = \mathbf{u}(t^+) + \boldsymbol{\lambda}_v(t^+) + \mu(t^+)\mathbf{r}(t),
\end{equation}
which, by Theorem \ref{thm:continuity}, implies
\begin{equation}
\boldsymbol{\lambda}_v(t^-) = \boldsymbol{\lambda}_v(t^+) + \mu(t^+)\mathbf{r}(t).
\end{equation}
Thus, by \eqref{eq:lvJumpS},
\begin{equation} \label{eq:pi2}
    \pi_2 = \frac{\mu(t^+)}{2} = \frac{a^2}{2D^2}.
\end{equation}
Combining the time derivative of \eqref{eq:el-u} with \eqref{eq:el-lv} and substituting it along with \eqref{eq:pi2} into \eqref{eq:lpJumpS} implies
    $\dot{\mathbf{u}}(t^-) = \dot{\mathbf{u}}(t^+) + 2D\pi_1\mathbf{r}(t),$
which, projected onto $\hat{q}$ yields
\begin{align} \label{eq:uDotQ}
    \dot{\mathbf{u}}(t^-)\cdot\hat{q} &= \dot{\mathbf{u}}(t^+)\cdot\hat{q}.
\end{align}

Thus, we have three unknowns: the angle of $\mathbf{r}(t)$, the relative speed $a = ||\dot{\mathbf{r}}(t)||$, and the optimal transition time $t$.
The three corresponding equations are continuity in $\mathbf{u}(t)$ by Theorem \ref{thm:continuity} and the continuity of $\dot{\mathbf{u}}(t)\cdot\hat{q}$ by \eqref{eq:uDotQ}.
This is sufficient to determine $\mathbf{r}(t)$ and $\dot{\mathbf{r}}(t)$ using the tangency conditions \eqref{eq:tangencyFixed}, which determines the trajectory of the agent along the fixed distance constraint.

Finally, we note that the in a swarm system, for two agents $i, j\in\mathbb{N}$, we may write a safe distance constraint for agent $i$ relative to agent $j$,
\begin{equation}
g_{ij}\big(\mathbf{x}(t), t\big) = ||\mathbf{p}_i(t) - \mathbf{p}_j(t)||^2 \geq \big(2R\big)^2,    
\end{equation}
for some agent radius $R$.
This results in tangency conditions identical to \eqref{eq:tangencyFixed} when the constraint is active, i.e., $\mathbf{N}\big(\mathbf{x}(t), t\big) = 0$.
Thus, the preceding analysis holds for the transition to the collision avoidance constraint if $j$'s trajectory satisfies Assumption \ref{smp:uddot}.
%In both cases, the agent must enter and maintain a circular constraint centered on a point of interest.

\section{Conclusions and Future Work} \label{sec:conclusion}

In this paper, we presented a proof of continuity of the control input for a class of energy-minimizing systems when transitioning between constrained and unconstrained trajectories.
We extended this result to include the case where the constraint becomes discontinuous at a point, and finally, we proposed an original barycentric motion for constraint-driven optimal control.
We derived the optimal control input when traveling along the constrained arc and derived the optimality conditions for transitioning to the constrained arc as a function of the and control variables.
Finally, we extended our analysis to include collision avoidance between agents.

Ongoing research addresses the potential of deriving an efficient shooting method to numerically solve the proposed jump conditions in real-time.
Future work should consider extending Theorem \ref{thm:continuity} for the cases where (1) multiple constraints become active simultaneously, (2)  the agent transitions between different constrained arcs, (3) only the right partial derivatives of $\mathbf{N}(\mathbf{x}(t), t)$ exist, and (4) a constraint becomes active only at a single instant in time.
Another potential direction for future research is to include additional agent interactions to achieve a desired emergent behavior from the system \cite{Beaver2020BeyondFlocking}.

\bibliographystyle{IEEEtran}
\bibliography{mendeley, IDS_Pubs}

\end{document}